\documentclass[amsart]{amsart}
\usepackage{amsmath,amssymb,amsthm}
\usepackage{mathrsfs}
\usepackage{graphicx}
\usepackage{tikz-cd}
\usepackage{enumerate}
\usepackage{cite}
\usepackage{url}
\usepackage{amsfonts}
\usepackage{latexsym}
\usepackage{enumitem}
\usepackage{mathrsfs}
\usepackage[all]{xy} \SelectTips{eu}{}
\usepackage{enumerate}
\usepackage{hyperref}
\usepackage{color}

\numberwithin{equation}{section}
\newtheorem{theorem}{Theorem}[section]
\newtheorem{lemma}[theorem]{Lemma}
\newtheorem{proposition}[theorem]{Proposition}
\newtheorem{corollary}[theorem]{Corollary}
\newtheorem{definition}[theorem]{Definition}
\newtheorem{remark}[theorem]{Remark}

\newcommand{\Coker}{\operatorname{Coker}}
\newcommand{\Tor}{\operatorname{Tor}}
\newcommand{\Ext}{\operatorname{Ext}}
\newcommand{\Hom}{\operatorname{Hom}}

\newcommand{\pd}{\operatorname{pd}}

\newcommand{\Mod}{\operatorname{Mod}}
\newcommand{\Ker}{\operatorname{Ker}}
\newcommand{\Img}{\operatorname{Im}}

\newcommand{\A}{\mathcal{A}}
\newcommand{\B}{\mathcal{B}}
\newcommand{\C}{\mathcal{C}}
\newcommand{\W}{\mathcal{W}}
\newcommand{\T}{\mathcal{T}}
\newcommand{\X}{\mathcal{X}}

\def\Id{\mathrm{Id}}
\def\add{\mathrm{add}}
\def\Add{\mathrm{Add}}
\newcommand{\Proj}{\operatorname{Proj}}
\newcommand{\GP}{\operatorname{GP}}

\begin{document}

\title[Tilting pairs and Wakamatsu tilting pairs over cleft extensions]
{Tilting pairs and Wakamatsu tilting pairs of subcategories over cleft extensions}
\author{Guoqiang Zhao}
\address{Department of Mathematics, Hangzhou Dianzi University, Hangzhou 310018, China}
\email{gqzhao@hdu.edu.cn}

\author{Juxiang Sun}
\address{School of Mathematics and Statistics, Shangqiu Normal University, Shangqiu 476000, China}
\email{Sunjx8078@163.com}

\begin{abstract} 
Let $(\mathcal{B},\mathcal{A}, i, e, l)$ be a cleft extension of abelian categories.
We prove that the functor $l$ preserves and reflects (Wakamatsu) tilting pairs of subcategories under certain conditions, 
unifying an abundance of known results.
Then, we apply our results to the cleft extensions of module categories,  
and give characterizations of tilting pairs and Wakamatsu tilting pairs over $\theta$-extension of rings and tensor rings,
which not only recover the earlier results in this direction, but also obtain some new conclusions.
\end{abstract}

\maketitle



\section{Introduction}

Tilting module and tilting theory play an important role in the representation theory of Artin algebras.
The notion of the classical tilting module over an Artin algebra was introduced by Brenner and Butler in \cite{BB79},
whose projective dimension is one.
Afterwards they were extended to arbitrary modules of finitely projective dimension over associative rings \cite{HC01}.
Meanwhile, Miyashita \cite{Mi01} introduced the notion of tilting pairs,
which can be regarded as a common generalization of tilting modules and cotilting modules. 
Tilting pairs are useful for constructing tilting modules associated with a series of idempotent ideals.
Recently, Zhao, Zhu and Zhuang \cite{ZZZ21} introduced the concept of tilting pairs of subcategories in extriangulated categories, 
and obtained some characterizations of tilting pairs.

A further generalization of the tilting modules to modules of possibly infinite projective dimension was studied in \cite{Wa90, Wa04}, 
which is now called the Wakamatsu tilting module. Many results on tilting modules were naturally generalized.
As a categorical analogue of Wakamatsu tilting modules, the notion of Wakamatsu tilting subcategories was introduced by Enomoto in \cite{E17},
via which the author classifies exact categories.
Many important properties of these notions mentioned above were presented in the past few years.
One of the most important tasks in this topic is to describe
them over a given extended ring. In a series of papers,
these notions over triangular matrix rings, trivial ring extensions and Morita context rings are investigated,
such as \cite{Mi85, M23, ZMZ23, ZW19, ZW26} and the references in lectures.

On the other hand, 
trivial extensions, tensor rings and, more generally, positively graded rings and $\theta$-extensions \cite{Ma93}
can be unified by the concept of cleft extension of rings.
The study of cleft extension of rings is quite essential and has interesting applications to classical ring theory.
As a natural generalization of cleft extension of rings and
trivial extension of abelian categories,
the notion of cleft extension of abelian categories was
introduced by Beligiannis in \cite{Bel00},
which provides a general setting and unified treatment for the study of cleft extension of rings.
Moreover, the framework of cleft extension has been proven to be very useful in reducing some homological properties of rings, 
such as the finitistic dimensions \cite{EGPS22, GPS21},
injective generations \cite{KP25}, the Igusa-Todorov distances and
extension dimensions \cite{MZL25}, and
the Gorenstein projective modules, singularity categories and Gorensteinness \cite{K26}.


The main aim of the present paper is to investigate tilting pairs and Wakamatsu tilting pairs of subcategories 
over cleft extensions of abelian categories. The paper is organized as follows:

In Section 2, we recall some basic concepts and facts.

Let $(\mathcal{B},\mathcal{A}, i, e, l)$ be a cleft extension.
In Section 3, we prove that under some condition the functor $l$ preserves and reflects tilting pairs of subcategories of $\B$ (see Theorems 3.5).
As a consequence, we give a characterization when $l$ preserves and reflects tilting objects of $\B$.

In Section 4, we first introduce the concept of Wakamatsu tilting pairs of subcategories in terms of 
two important subcategories $\X_{\W}$ and ${_\W}\X$. Then we give sufficient conditions when $l$ preserves $\X_{\W}$ and ${_\W}\X$.
As applications, we obtain that $l$ preserves Wakamatsu tilting pairs of subcategories (see Theorems 4.5),
as well as Gorenstein categories and Gorenstein projective objects under a mild situation. 
When the converse directions hold true are also discussed (see Theorems 4.9).
As a consequence, we give a decription when $l$ preserves and reflects Wakamatsu tilting subcategories,
which recovers the known results over trivial extensions.

Section 5 is devoted to applications in cleft extension of module categories. In particular,
we obtain transfer properties of tilting pairs and Wakamatsu tilting pairs of subcategories of modules over $\theta$-extensions and tensor rings. 
We not only recover some known results over trivial ring extensions, including triangular matrix rings and Morita context rings, 
but also obtain new conclusions over these ring extensions.

\section{Preliminaries and notations}

Throughout this paper, we always assume that $\A$ and $\B$ are abelian categories with enough projective and injective objects. 
For an object $X$ in $\A$, denote by $\pd_{\A}(X)$ the projective dimensions of $X$.
We denote by $\Add X$ (resp. $\add X$) the full subcategory of $\A$ 
consisting of all direct summands of (resp. finite) direct sums of copies of $X$.

\begin{lemma} \label{lem:2.1}
Suppose that the functor $\mathbb{G}: \A\rightarrow\B$ has a left adjoint functor $\mathbb{F}$. 
 Let $\nu\colon\mathsf{Id}_{\mathcal{B}}\rightarrow\mathbb{GF}$ be the unit and 
$\mu\colon\mathbb{FG}\rightarrow \Id_{\mathcal{A}}$ the counit of the adjoint pair $(\mathbb{F}, \mathbb{G})$.
If $\mathbb{G}$ is faithful exact, then, for every $A\in\mathcal{A}$,
$\mu_{A}\colon \mathbb{FG}(A)\rightarrow A$ is an epimorphism. 
\end{lemma}

\begin{proof}
For every $A\in\mathcal{A}$ and $B\in\mathcal{B}$, the following relations are satisfied: 
 \[
 \mathsf{Id}_{\mathbb{F}(B)}=\mu_{\mathbb{F}(B)}\mathbb{F}(\nu_B) \  \text{ and } \  \mathsf{Id}_{\mathbb{G}(A)}=\mathbb{G}(\mu_A)\nu_{\mathbb{G}(A)}.
 \]
which implies that $\mathbb{G}(\mu_{A})$ is a split epimorphism.
Since $\mathbb{G}$ is faithful exact, it follows that $\mu_{A}\colon \mathbb{FG}(A)\rightarrow A$ is an epimorphism for every $A\in\mathcal{A}$.
\end{proof}

\begin{definition}\textnormal{(\!\!\!\cite{Bel00})}
{\rm 
A {\it cleft extension} of an abelian category $\mathcal{B}$
is an abelian category $\mathcal{A}$
together with functors $$\xymatrix@!=4pc{ \mathcal{B} \ar[r]|{i} & \mathcal{A}
			 \ar[r]|{e} & \mathcal{B}
			\ar@/_1pc/[l]|{l} }  $$
such that the functor $e$ is faithful exact and admits a left adjoint $l$, and there is a natural isomorphism
$ei \simeq \Id_{\mathcal{B}}$.
}
\end{definition}

From now on we will denote a cleft extension by $(\mathcal{B},\mathcal{A}, i, e, l)$, and it
gives rise to additional properties. For instance, it follows that the functor $i$ is fully
faithful and exact (see \cite[Lemma 2.2]{Bel00}, \cite[Lemma 2.2(ii)]{GPS21}).  Moreover, there is a
functor $q: \mathcal{A} \rightarrow \mathcal{B}$, which is left adjoint of $i$ (see \cite[Proposition 2.3]{Bel00}
and \cite[Lemma 2.2(iv)]{GPS21}). Then, $(ql, ei)$ is an adjoint pair and since $ei \simeq \Id_{\mathcal{B}}$,
it follows that $ql \simeq \Id_{\mathcal{B}}$. We mention that the functors $l$ and $q$ are both right exact and preserve coproduct.

Denote by $\nu\colon\mathsf{Id}_{\mathcal{B}}\rightarrow el$ the unit and by $\mu\colon le\rightarrow \mathsf{Id}_{\mathcal{A}}$ the counit of the adjoint pair $(l, e)$. 
It follows from Lemma 2.1 that 
there is an exact sequence of functors
$$0\rightarrow G\rightarrow le\xrightarrow{\mu} \Id_{\mathcal{A}}\rightarrow 0,\qquad(2.1)$$
where $G$ is a right exact endofunctor of $\A$ defined by $G(A) = \ker \mu_{A}$.


Because $e(\mu_{i(B)})$ is a split epimorphism for every $B\in\mathcal{B}$ and $ei \cong \Id_{\mathcal{B}}$, one has a split exact sequence
$$0\rightarrow F\rightarrow el\rightarrow \Id_{\mathcal{B}}\rightarrow 0,\qquad(2.2)$$
where $F= eGi$ is a right exact endofunctor of $\B$ ( for details see \cite[Section 2]{GPS21}).

From \cite[Lemma 2.4]{GPS21} we know that $F^ne \cong eG^n$ for $n \geq 1$, thus in particular $F$ is nilpotent if and only if $G$ is nilpotent. 
These two endofunctors $F$ and $G$ are of central importance in cleft extensions, and we fix them throughout the paper.

\begin{lemma} \label{lem:2.3} \textnormal{(\!\!\!\cite[Lemma 2.3]{K26})}
Let $(\mathcal{B},\mathcal{A}, i, e, l)$ be a cleft extension.
$A\in\A$ is projective if and only if $A$ is a direct summand of $l(P)$ for some $P\in\Proj(\B)$.
\end{lemma}

\begin{lemma}\label{lem:2.4} \textnormal{(\!\!\!\cite[Lemma 2.2]{K26})} 
Let $(\mathcal{B},\mathcal{A}, i, e, l)$ be a cleft extension and $X\in\B$. 
The following hold: 
    \begin{itemize}
     \item[(1)] $X\in \Proj\mathcal{B}$ if and only if $\mathsf{l}(X)\in\Proj\mathcal{A}$.
        \item[(2)] $\mathbb{L}_{j}F(X)\cong e\mathbb{L}_{j}l(X)$ for all $j\geq 1$, where $\mathbb{L}_{j}$ is the $j$th left derived functor. 
        \item[(3)] If $\mathbb{L}_{j}F(X)=0$ for all $j\geq 1$, then 
        $\Ext_{\mathcal{A}}^j(l(X),Y)\cong \mathsf{Ext}_{\mathcal{B}}^j(X,e(Y))$ for all $j\geq 1$ and every object $Y$ of $\mathcal{A}$.
    \end{itemize} 
\end{lemma} 

Let $\W$ be a subcategory of $\B$. We write
\begin{center}
${}^{\bot}\W= \{B\in\B | \Ext^{\geqslant1}_{\B}(B, W)= 0$ for any $W\in\W\}.$
\end{center}
$\W$ is called self-orthogonal if $\W\subseteq {}^{\bot}\W$.

We use $\X_{\W}$ to denote the class of objects $X\in {}^{\bot}\W$ such that there exists an exact sequence
$$0\rightarrow X\xrightarrow{d_{0}} W_{0}\xrightarrow{d_{1}} W_{1}\rightarrow\cdots $$
with $W_{k} \in \W$ and $\Img(d_{k}) \in {}^{\bot}\W$ for each $k\geq 1$.
Dually, ${\W}^{\bot}$ and ${_\W}\X$ can be defined.

We mention that for the case of $\W = \add T$, where $T$ is a finitely generated and self-orthogonal module, 
the important subcategory $\X_{\W}$ has been studied by Auslander and Reiten \cite{AR91} and by Mantese and Reiten \cite{MR04}.

Recall from \cite{SZH20} that the right Gorenstein subcategory $\mathbf{rG}(\W)$ of $\B$ consists of objects $X\in {}^{\bot}\W$ 
such that there exists an exact sequence
$$0\rightarrow X\xrightarrow{d_{0}} W_{0}\xrightarrow{d_{1}} W_{1}\rightarrow\cdots $$
with all $W_{k} \in \W$ and it remains exact after applying $\Hom_{\B}(-, W)$ for any $W\in\W$.

In general, $\X_{\W}\subseteq \mathbf{rG}(\W)$. In case $\W$ is self-orthogonal, 
then $\X_{\W} =\mathbf{rG}(\W)$ by \cite[Remark 3.1]{ZW26}.

\section{Tilting pairs of subcategories over cleft extensions}


Let $\T$ be a subcategory of $\B$ and $n$ an integer. We denote by
$\widetilde{\T}_{n}$ the subcategory of all $B\in\B$ such that there is an exact sequence 
$0\rightarrow B\to T_{0}\to\cdots\to T_{n}\to 0$ with each $T_{i}\in\T$.
Dually, $\widehat{\C}_{n}$ denotes the subcategory of all $B\in\B$ such that there is an exact sequence
$ 0 \to C_{n}\to\cdots\to C_{0}\to B\to 0 $
with every $C_{i}\in\C$.

We call an object $T\in\B$ \emph{$n$-tilting} if (T1) $\pd_{\B}T\leq n$; (T2) $\Ext_{\B}^j(T$, $T^{(I)})=0$ for any set $I$ and $1\leq j\leq n$;
and (T3) for any $P\in\Proj(\B)$, there exists an exact sequence 
$$0\to P\to T_{0}\to\cdots\to T_{n}\to 0 $$ with each $T_{k}\in\Add T$.

Let $\C$ and $\T$ be two self-orthogonal subcategories of $\B$. 
$(\C, \T)$ is called an $n$-tilting pair if 
(1) $\C\subseteq \widetilde{\T}_{n}$; (2) $\T\subseteq \widehat{\C}_{n}$.
In this case, $\T$ is called a $n$-$\C$-tilting subcategory and $\C$ is called a $n$-$\T$-cotilting subcategory.

\begin{remark}
{\rm (1) Let $\C = \Proj(\B)$, $\T$ = Add$(T)$ for some $T\in\B$. 
Then $\T$ is an $n$-$\C$-tilting subcategory if and only if 
$T$ is an $n$-tilting object in $\B$.

(2) Let $R$ be an Artin algebra and $\B=$ mod $R$. 
If $\C = \add(C)$, $\T = \add(T)$, then $(\C, \T)$ is an $n$-tilting pair if and only if 
$(C, T)$ is an $n$-tilting pair of $R$-modules (see \cite{Mi01}).}
\end{remark}

\begin{proposition} \label{prop:3.2}
Let $(\mathcal{B},\mathcal{A},i,e,l)$ be a cleft extension and $(\C, \T)$ be an $n$-tilting pair in $\B$.
If $F(\C)\subseteq \C^{\bot}$, $F(\T)\subseteq \T^{\bot}$ and $\mathbb{L}_{j}F(T)=0$ for any $T\in\T$ and $j\geq 1$. 
Then $(l(\C), l(\T))$ is an $n$-tilting pair in $\A$.
\end{proposition}

\begin{proof}
Since $\C\subseteq \widetilde{\T}_{n}$, then for any $C\in\C$, 
there is an exact sequence 
$$0\rightarrow C\to T_{0}\to T_{1}\to\cdots\to T_{n}\to 0, \qquad(3.1)$$
with each $T_{k}\in\T$. Because $\mathbb{L}_{j}F(T)=0$ for any $T\in\T$ and $j\geq 1$,
one has $\mathbb{L}_{j}l(T)=0$ for $T\in\T$ by Lemma \ref{lem:2.4}(2) and the faithfulness of $e$. 
Thus applying $l$ to the exact sequence (3.1) gives rise to an exact sequence
$$0\to l(C)\to l(T_{0})\to l(T_{1})\to\cdots\to l(T_{n})\to 0$$
since $l$ is right exact. That is $l(\C)\subseteq \widetilde{l(\T)}_{n}$.

Since $\mathbb{L}_{j}F(T)=0$ for any $T\in\T$ and $j\geq 1$, from the exact sequence (3.1) one has 
$\mathbb{L}_{j}F(C)\cong\mathbb{L}_{j+n}F(T_{n})=0$ for $j\geq 1$ by dimension shifting.
Thus for any $C_{1}, C_{2}\in\C$, it follows from Lemma \ref{lem:2.4}(3) that 
$\Ext_{\A}^j(l(C_{1}), l(C_{2}))\cong$ 
$\Ext_{\B}^j(C_{1}, el(C_{2}))\cong$$\Ext_{\B}^j(C_{1}$, $C_{2}\oplus F(C_{2}))=0$
since $\C$ is self-orthogonal and $F(\C)\subseteq \C^{\bot}$.
Hence $l(\C)$ is self-orthogonal.
Similarly, one has that $l(\T)$ is also self-orthogonal.

On the other hand, since $\T\subseteq \widehat{\C}_{n}$, then for any $T\in\T$, there is an exact sequence 
$$ 0 \to C_{n}\to \cdots\to C_{1}\to C_{0}\to T\to 0 \qquad(3.2)$$
with each $C_{k}\in\C$.
Note that $\mathbb{L}_{j}l(X)=0$ for $X\in\C$ and $\T$.
Applying $l$ to the exact sequence (3.2) gives rise to an exact sequence
$$  0 \to l(C_{n})\cdots\to l(C_{1})\to l(C_{0})\to l(T)\to 0,$$
which implies that $l(\T)\subseteq \widehat{l(\C})_{n}$.
Therefore, $(l(\C), l(\T))$ is an $n$-tilting pair.
\end{proof}

\begin{lemma}\label{lem:3.3} 
Let $(\mathcal{B},\mathcal{A},i,e,l)$ be a cleft extension and $X\in\B$. 
 If $\mathbb{L}_{j}l(X)=0$ for all $j\geq 1$, then $\mathbb{L}_{j}q(l(X))=0$ for all $j\geq 1$.
 \end{lemma} 
 
 \begin{proof}
 Let $ \cdots\to P_{1}\to P_{0}\to X\to 0 $ be a projective resolution of $X$.
 Because $l$ is right exact and $\mathbb{L}_{j}l(X)=0$ for all $j\geq 1$,
 applying $l$ to the exact sequence above gives rise to an exact sequence
 $$ \cdots\to l(P_{1})\to l(P_{0})\to l(X)\to 0 $$
with all $l(P_{i})\in\Proj(\A)$ by Lemma \ref{lem:2.4}(1).
Thus one obtains a complex $$ \cdots\to ql(P_{1})\to ql(P_{0})\to ql(X)\to 0, $$
which is exact since $ql\cong\Id_{\B}$.
This implies that $\mathbb{L}_{j}q(l(X))=0$ for all $j\geq 1$.
 \end{proof}

\begin{proposition} \label{prop:3.4}
Let $(\mathcal{B},\mathcal{A},i,e,l)$ be a cleft extension and $\C, \T$ be subcategories of $\B$ 
such that $\mathbb{L}_{j}F(T)=0$ for any $T\in\T$ and  $j\geq 1$.
If $(l(\C), l(\T))$ is an $n$-tilting pair in $\A$,
then $(\C, \T)$ is an $n$-tilting pair in $\B$ and $F(\C)\subseteq\C^{\bot}$, $F(\T)\subseteq\T^{\bot}$.
\end{proposition}

\begin{proof}
Since $l(\C)\subseteq \widetilde{l(\T})_{n}$, then for any $C\in\C$, 
there is an exact sequence 
$$\bigtriangleup: 0\to l(C)\to l(T_{0})\to l(T_{1})\to\cdots \to l(T_{n})\to 0$$
with $T_{k}\in\T$ for each $0\leq k\leq n$. 
Since $\mathbb{L}_{j}F(T)=0$ for any $T\in\T$ and  $j\geq 1$, then 
$\mathbb{L}_{j}l(T)=0$ by Lemma \ref{lem:2.4}(2) and the faithfulness of $e$,
and hence $\mathbb{L}_{j}q(l(T))=0$ by Lemma \ref{lem:3.3}. 
Thus, by applying $q$ to $\bigtriangleup$, one gets an exact sequence
$$q(\bigtriangleup): 0\to C\to T_{0}\to T_{1}\to\cdots \to T_{n}\to 0,$$
which shows that $\C\subseteq \widetilde{\T}_{n}$.
Moreover, since $F$ is right exact and  $\mathbb{L}_{j}F(T)=0$ for any $T\in\T$ and $j\geq 1$, 
then by applying $F$ to $q(\bigtriangleup)$ we have $\mathbb{L}_{j}F(C)=0$ for any $C\in\C$ and $j\geq 1$.
Thus one can prove  $\T\subseteq \widehat{\C}_{n}$ by a similar argument.

For any $C_{1}, C_{2}\in\C$, 
it follows from Lemma \ref{lem:2.4}(3) that 
$0=\Ext_{\A}^i(l(C_{1}), l(C_{2}))\cong$ 
$\Ext_{\B}^i(C_{1}, el(C_{2}))\cong$$\Ext_{\B}^i(C_{1}$, $C_{2}\oplus F(C_{2}))$.
Thus $\Ext_{\B}^i(C_{1}$, $C_{2})=0 =\Ext_{\B}^i(C_{1}$, $F(C_{2}))$,
which means that $\C$ is self-orthogonal and $F(\C)\subseteq\C^{\bot}$.
Similarly, one has that $\T$ is self-orthogonal and $F(\T)\subseteq\T^{\bot}$.
\end{proof}

Combining Proposition \ref{prop:3.2} and \ref{prop:3.4}, we get the main result in this section.

\begin{theorem} \label{thm:3.5}
Let $(\mathcal{B},\mathcal{A}, i, e, l)$ be a cleft extension and $\C, \T$ be two subcategories of $\B$ 
such that $\mathbb{L}_{j}F(T)=0$ for any $T\in\T$ and $j\geq 1$. Then the following conditions are equivalent.

(1) $(\C, \T)$ is an $n$-tilting pair in $\B$ and $F(\C)\subseteq \C^{\bot}$, $F(\T)\subseteq \T^{\bot}$. 

(2) $(l(\C), l(\T))$ is an $n$-tilting pair in $\A$.
\end{theorem}

\begin{corollary} \label{cor:3.6}
Let $(\mathcal{B},\mathcal{A},i,e,l)$ be a cleft extension and $T\in \B$ 
such that $\mathbb{L}_{j}F(T)=0$ for any $j\geq 1$. 
If $l(\Proj(\B))=\Proj(\A)$, then the following conditions are equivalent.

(1) $T$ is an $n$-tilting object in $\B$ and $\Ext_{\B}^j(T$, $F(T^{(I))})=0$ for any set $I$ and $j\geq 1$;. 

(2) $l(T)$ is an $n$-tilting object in $\A$.
\end{corollary}

\begin{proof}
$T$ is an $n$-tilting object in $\B$ if and only if $(\Proj(\B), \Add T)$ is an $n$-tilting pair in $\B$.
Since $F(\Add T)\subseteq (\Add T)^{\bot}$ by hypothesis, this is equivalent to that 
$(l(\Proj(\B)), l(\Add T))$ is an $n$-tilting pair in $\A$ by Theorem \ref{thm:3.5}.
Note that $l$ preserve coproduct, then  $l(\Add T) =\Add l(T)$,
which means that $l(T)$ is an $n$-tilting object in $\A$.
\end{proof}

\begin{remark}\label{rmk:3.7}
{\rm (1) If $\Ker q =0$,  then the condition $l(\Proj(\B))=\Proj(\A)$ in Corollary \ref{cor:3.6} is satisfied.
Indeed, Let $Q\in\Proj(\A)$. Lemma \ref {lem:2.3} infers that $Q$ is a direct summand of $l(P)$ for some $P\in\Proj(\B)$, 
and so $q(Q)$ is a direct summand of $P$.
Notice that $q(Q-lq(Q))=q(Q)-qlq(Q)=0$, thus $Q-lq(Q)\in \Ker q =0$, and hence $Q=lq(Q)\in$ $l(\Proj(\B))$.

(2) Let $R\ltimes M$ be the  trivial extension of a ring $R$ by an $R$-$R$-bimodule $M$.
Then $\Mod R\ltimes M$ is a cleft extension of the module category $\Mod R$.
It follows from \cite[Corollary 1.6]{FGR75} that $l(\Proj(R))=\Proj(R\ltimes M)$.}
\end{remark}

\section{Wakamatsu tilting pairs of subcategories over cleft extensions}

\begin{definition}\label{def:4.1}
Let $\C$ and $\W$ be two self-orthogonal subcategories of $\B$. 
$(\C, \W)$ is called a Wakamatsu tilting pair if 

(1) $\C\subseteq \X_{\W}$;

(2) $\W\subseteq{_\C}\X$.
\end{definition}

\begin{remark}\label{rmk:4.2}
{\rm (1) $(\Proj(\B), \W)$ is a Wakamatsu tilting pair if and only if 
$\W$ is a Wakamatsu tilting subcategory defined in \cite[Definiton 3.2]{ZW26}.

(2) Let $R$ be an associative ring and $\C = \add(R)$, $\W = \add(\omega)$ for some left $R$-module $\omega$. 
Then $(\add(R), \add(\omega))$ is a Wakamatsu tilting pair if and only if 
$\omega$ is a Wakamatsu tilting module (see\cite{MR04} for details).}
\end{remark}

\begin{lemma} \label{lem:4.3}
Let $(\mathcal{B},\mathcal{A},i,e,l)$ be a cleft extension and $\W$ be a subcategory of $\B$.
Suppose that $F(\W)\subseteq \W$ and $\mathbb{L}_{j}F(X)=0$ for any $X\in{^\bot}\W$ and $j\geq 1$. 
If $\C\subseteq \X_{\W}$, then $l(\C)\subseteq \X_{l(\W)}$.
\end{lemma}

\begin{proof}
Since $\C\subseteq \X_{\W}$, then for any $C\in\C$, 
there is an exact sequence 
$$0\rightarrow C\xrightarrow{f_{0}} W_{0}\xrightarrow{f_{1}} W_{1}\rightarrow\cdots \qquad(4.1)$$
with $W_{k}\in\W$ and $\Img f_{k}\in {^\bot}\W$ for each $k\geq 0$.
Thus $\mathbb{L}_{j}F(\Img f_{k})=0$ for $j\geq 1$ by assumption,
and hence $\mathbb{L}_{j}l(\Img f_{k})=0$ by Lemma \ref{lem:2.4}(2) and the faithfulness of $e$. 
Note that $l$ is right exact, applying $l$ to the exact sequence (4.1) gives rise to an exact sequence
$$0\rightarrow l(C)\xrightarrow{l(f_{0})} l(W_{0})\xrightarrow{l(f_{1})} l(W_{1})\rightarrow\cdots .$$
Hence $\Img l(f_{k})\cong l(\Img f_{k})$.
It follows from Lemma \ref{lem:2.4}(3) that 
$$\Ext_{\A}^j(l(\Img f_{k}), l(W))\cong
\Ext_{\B}^j(\Img f_{k}, el(W))\cong\Ext_{\B}^j(\Img f_{k}, W\oplus F(W))=0$$
since $\Img f_{k}\in {^\bot}\W$ and $F(\W)\subseteq \W$.
Therefore, $\Img l(f_{k})\in {^\bot}l(\W)$, and so $l(\C)\subseteq \X_{l(\W)}$.
\end{proof}

\begin{lemma} \label{lem:4.4}
Let $(\mathcal{B},\mathcal{A}, i, e, l)$ be a cleft extension and $\C$ be a subcategory of $\B$.
Suppose that $F(\C^{\bot})\subseteq \C^{\bot}$ and $\mathbb{L}_{j}F(X)=0$ for any $X\in\C^{\bot}$ and $j\geq 1$. 
If $\W\subseteq{_\C}\X$, then $l(\W)\subseteq {_{l(\C)}}\X$.
\end{lemma}

\begin{proof}
Because $\W\subseteq{_\C}\X$, then $\W\subseteq \C^{\bot}$, and for any $W\in\W$, there is an exact sequence 
$$ \cdots\rightarrow C_{1}\xrightarrow{g_{1}} C_{0}\xrightarrow{g_{0}} W\to 0 \qquad(4.2)$$
with $C_{k}\in\C$ and $\Img g_{k}\in \C^{\bot}$ for each $k\geq 0$. Note that 
$\mathbb{L}_{j}l(X)=0$ for any $X\in\C^{\bot}$ by assumption and Lemma \ref{lem:2.4}(2).
Applying $l$ to the exact sequence (4.2) gives rise to an exact sequence
$$ \cdots\rightarrow l(C_{1})\xrightarrow{l(g_{1})} l(C_{0})\xrightarrow{l(g_{0})} l(W)\to 0.$$
Thus $\Img l(g_{k})\cong l(\Img g_{k}).$
Given $C\in\C$, from Lemma \ref{lem:2.4}(3) one has
$$\Ext_{\A}^j(l(C), l(\Img g_{k}))\cong
\Ext_{\B}^j(C, el(\Img g_{k}))\cong\Ext_{\B}^j(C, \Img g_{k}\oplus F(\Img g_{k}))=0$$
since $\Img g_{k}\in \C^{\bot}$ and $F(\C^{\bot})\subseteq \C^{\bot}$.
Therefore, $\Img l(g_{k})\in l(\C)^{\bot}$, and hence $l(\W)\subseteq {_{l(\C)}}\X$, as desired.
\end{proof}

\begin{theorem} \label{thm:4.5}
Let $(\mathcal{B},\mathcal{A},i,e,l)$ be a cleft extension and $(\C, \W)$ be a Wakamatsu tilting pair in $\B$.
If $F(\C^{\bot})\subseteq \C^{\bot}$, $F(\W)\subseteq \W$ and $\mathbb{L}_{j}F(X)=0$ for any $X\in{^\bot}\W$ and $j\geq 1$, 
then $(l(\C), l(\W))$ is a Wakamatsu tilting pair in $\A$.
\end{theorem}

\begin{proof}
Because $\C\subseteq \X_{\W}$, one has that $\C\subseteq{^\bot}\W$.
For any $C_{1}, C_{2}\in\C$, since $\mathbb{L}_{j}F(C_{1})=0$ for $j\geq 1$ by assumption,
it follows from Lemma \ref{lem:2.4}(3) that 
$$\Ext_{\A}^j(l(C_{1}), l(C_{2}))\cong
\Ext_{\B}^j(C_{1}, el(C_{2}))\cong\Ext_{\B}^j(C_{1}, C_{2}\oplus F(C_{2}))=0$$
since $\C$ is self-orthogonal and $F(\C)\subseteq F(\C^{\bot})\subseteq\C^{\bot}$.
Thus $l(\C)$ is self-orthogonal.
Similarly, one has that $l(\W)$ is also self-orthogonal.

Note that $\W\subseteq{^\bot}\W$ since $\W$ is self-orthogonal.
The assertion follows from Lemma \ref{lem:4.3} and \ref{lem:4.4}.
\end{proof}

\begin{corollary} \label{cor:4.6}
Let $(\mathcal{B},\mathcal{A}, i, e, l)$ be a cleft extension and $\W$ be a Wakamatsu tilting subcategory of $\B$.
If $l(\Proj(\B))=\Proj(\A)$, $F(\W)\subseteq \W$ and $\mathbb{L}_{j}F(X)=0$ for any $X\in{^\bot}\W$ and $j\geq 1$. 
Then $l(\W)$ is a Wakamatsu tilting subcategory of $\A$.
\end{corollary}

\begin{proof}
Because $\W$ is a Wakamatsu tilting subcategory of $\B$ if and only if $(\Proj(\B), \W)$ is a Wakamatsu tilting pair in $\B$.
From Theorem \ref{thm:4.5} we know that $(l(\Proj(\B)), l(\W))$ is a Wakamatsu tilting pair in $\A$.
If $l(\Proj(\B))=\Proj(\A)$, then $l(\W)$ is a Wakamatsu tilting subcategory of $\A$.
\end{proof}

Next, we discuss when the converse of Theorem \ref{thm:4.5} holds true.

\begin{proposition} \label{prop:4.7}
Let $(\mathcal{B},\mathcal{A},i,e,l)$ be a cleft extension and $\C, \W$ be subcategories of $\B$
such that $\W$ is self-orthogonal and $\mathbb{L}_{j}F(X)=0$ for any $X\in\C$ and $j\geq 1$.
If $l(\C)\subseteq \X_{l(\W)}$ and $\mathbb{L}_{j}q(Y)=0$ for any $Y\in {^\bot}l(\W)$ and $j\geq 1$. 
Then $\C\subseteq \X_{\W}$.
\end{proposition}

\begin{proof}
Since $l(\C)\subseteq \X_{l(\W)}$, then $l(\C)\subseteq{^\bot}l(\W)$.
For any $C\in\C$, $W\in\W$, since $\mathbb{L}_{j}F(C)=0$ for $j\geq 1$ by assumption,
it follows from Lemma \ref{lem:2.4}(3) that 
$0=\Ext_{\A}^j(l(C), l(W))\cong$ 
$\Ext_{\B}^j(C, el(W))\cong$$\Ext_{\B}^j(C$, $W\oplus F(W))$.
Thus $\Ext_{\B}^j(C$, $W)=0$ for $j\geq 1$, 
which means that  $\C\subseteq{^\bot}\W$.

Since $l(\C)\subseteq \X_{l(\W)}$, then for any $C\in\C$, 
there is an exact sequence 
$$\bigtriangleup: 0\rightarrow l(C)\xrightarrow{h_{0}} l(W_{0})\xrightarrow{h_{1}} l(W_{1})\rightarrow\cdots $$
with $W_{k}\in\W$ and $\Img h_{k}\in {^\bot}l(\W)$ for each $k\geq 0$. 
Thus, for $j\geq 1$, $\mathbb{L}_{j}q(\Img h_{k})=0$ by assumption. 
Note that $q$ is right exact, applying $q$ to $\bigtriangleup$ gives rise to an exact sequence
$$q(\bigtriangleup): 0\rightarrow C\xrightarrow{q(h_{0})} W_{0}\xrightarrow{q(h_{1})} W_{1}\rightarrow\cdots $$

Because $\W$ is self-orthogonal and $\C\subseteq{^\bot}\W$, to prove every $\Img q(h_{k})\in {^\bot}\W$,
it suffices to show that $\Hom_{\B}(q(\bigtriangleup), \W)$ is exact. In fact,
for any $W\in\W$, from (2.1) we get an exact sequence
$$0\rightarrow Gi(W)\rightarrow l(W)\rightarrow i(W)\rightarrow 0\qquad(\ast)$$
Since $e$ is an exact functor, applying $e$ to $(\ast)$ gives rise to the exact sequence
$$0\rightarrow eGi(W)\rightarrow el(W)\rightarrow ei(W)\rightarrow 0\qquad(e(\ast))$$
Note that $eGi\cong Fei\cong F$, Thus the sequence $e(\ast)$ is split.
So the sequence  $\Hom_{\A}(l(C), \ast)\cong$ $\Hom_{\B}(C, e(\ast))$ and $\Hom_{\A}(l(W_{k}), \ast)\cong$ $\Hom_{\B}(W_{k}, e(\ast))$ 
are both split exact for each $k\geq 0$. Hence we get a split exact
sequence of complexes
$$0\rightarrow \Hom_{\A}(\bigtriangleup, Gi(W))\rightarrow \Hom_{\A}(\bigtriangleup, l(W))\rightarrow \Hom_{\A}(\bigtriangleup, i(W))\rightarrow 0,$$
which infers that $H^{n}(\Hom_{\A}(\bigtriangleup, i(W)))$ is a direct summand of $H^{n}(\Hom_{\A}(\bigtriangleup, l(W)))$.
It is clear that the middle term of the complexes above is exact, so is $\Hom_{\A}(\bigtriangleup, i(W))$.
Therefore, $\Hom_{\B}(q(\bigtriangleup), W)\cong$ $\Hom_{\A}(\bigtriangleup, i(W))$ is exact
since $(q, i)$ is adjoint, as desired.
\end{proof}

\begin{proposition} \label{prop:4.8}
Let $(\mathcal{B},\mathcal{A},i,e,l)$ be a cleft extension and $\C, \W$ be subcategories of $\B$
such that $\C$ is self-orthogonal and $\mathbb{L}_{j}F(X)=0$ for any $X\in\C\cup \W$ and $j\geq 1$.
If $l(\W)\subseteq {_{l(\C)}}\X$ and $\mathbb{L}_{j}q(Y)=0$ for any $Y\in l(\C)^{\bot}$ and $j\geq 1$. 
Then $\W\subseteq {_{\C}}\X$.
\end{proposition}

\begin{proof}
Because $l(\W)\subseteq {_{l(\C)}}\X$, $l(\W)\subseteq l(\C){^\bot}$.
For any $C\in\C$, $W\in\W$, since $\mathbb{L}_{j}F(C)=0$ for $j\geq 1$ by assumption,
it follows from Lemma \ref{lem:2.4}(3) that 
$0=\Ext_{\A}^j(l(C), l(W))\cong$ 
$\Ext_{\B}^j(C, el(W))\cong$$\Ext_{\B}^j(C$, $W\oplus F(W))$.
Thus $\Ext_{\B}^j(C$, $W)=0$ for $j\geq 1$, 
which means that  $\W\subseteq \C{^\bot}$.

Since $l(\W)\subseteq {_{l(\C)}}\X$, then for any $W\in\W$, 
there is an exact sequence 
$$\dag: \cdots\rightarrow l(C_{1})\xrightarrow{u_{1}} l(C_{0})\xrightarrow{u_{0}} l(W)\to 0$$
with $C_{k}\in\C$ and $\Img u_{k}\in l(\C){^\bot}$ for each $k\geq 0$. 
Thus, for $j\geq 1$, $\mathbb{L}_{j}q(\Img u_{t})=0$ by assumption. 
Note that $q$ is right exact, applying $q$ to $\dag$ gives rise to an exact sequence
$$\ddag: \cdots\rightarrow C_{1}\xrightarrow{q(u_{1})} C_{0}\xrightarrow {q(u_{0})} W\to 0 $$
Because $\C$ is self-orthogonal and $\W\subseteq\C^{\bot}$, to prove every $\Img q(u_{t})\in \C^{\bot}$,
it suffices to show that $\Hom_{\B}(\C, \ddag)$ is exact. Indeed,
for any $C\in\C$, it follows from (2.2) that there exists a split exact sequence of complexes 
$$0\rightarrow F(\ddag)\rightarrow el(\ddag)\rightarrow \ddag\rightarrow 0,$$
which induces a split exact sequence of complexes
$$0\rightarrow \Hom_{\B}(C, F(\ddag))\rightarrow \Hom_{\B}(C, el(\ddag))\rightarrow \Hom_{\B}(C, \ddag)\rightarrow 0.$$
Because $\mathbb{L}_{j}F(X)=0$ for any $X\in\C\cup \W$ and $j\geq 1$, 
one has $\mathbb{L}_{j}l(X)=0$ by Lemma \ref{lem:2.4}(2) and the faithfulness of $e$.
Thus the complex $l(\ddag)$ is exact, and hence
$\Hom_{\B}(C, el(\ddag))\cong\Hom_{\A}(l(C), l(\ddag))\cong\Hom_{\A}(l(C), \dag)$ is exact.
Therefore, $\Hom_{\B}(\C, \ddag)$ is exact, as desired.
\end{proof}

\begin{theorem} \label{thm:4.9}
Let $(\mathcal{B},\mathcal{A},i,e,l)$ be a cleft extension and $\C, \W$ be subcategories of $\B$
such that $\mathbb{L}_{j}F(X)=0$ for any $X\in\C\cup \W$ and $j\geq 1$.
If $(l(\C), l(\W))$ is a Wakamatsu tilting pair in $\A$ and $\mathbb{L}_{j}q(Y)=0$ for any $Y\in l(\C)^{\bot}\cup{^\bot}l(\W)$ and $j\geq 1$. 
Then $(\C, \W)$ is a Wakamatsu tilting pair in $\B$.
\end{theorem}

\begin{proof}
For any $C_{1}, C_{2}\in\C$, since $\mathbb{L}_{j}F(C_{1})=0$ for $j\geq 1$ by assumption,
it follows from Lemma \ref{lem:2.4}(3) that 
$\Ext_{\A}^j(l(C_{1}), l(C_{2}))\cong$ 
$\Ext_{\B}^j(C_{1}, el(C_{2}))\cong$$\Ext_{\B}^j(C_{1}$, $C_{2}\oplus F(C_{2}))$.
Thus $\Ext_{\B}^j(C_{1}$, $C_{2})=0$ for $j\geq 1$ since $l(\C)$ is self-orthogonal, 
which means that $\C$ is self-orthogonal.
Similarly, one has that $\W$ is also self-orthogonal.
Therefore, the conclusion is obtained by Proposition \ref{prop:4.7} and \ref{prop:4.8}.
\end{proof}

\begin{corollary} \label{cor:4.10}
Let $(\mathcal{B},\mathcal{A},i,e,l)$ be a cleft extension and $\W$ be subcategories of $\B$
such that $\mathbb{L}_{j}F(X)=0$ for any $X\in\W$ and $j\geq 1$.
If $l(\Proj(\B))=\Proj(\A)$ and $\mathbb{L}_{j}q(Y)=0$ for any $Y\in {^\bot}l(\W)$ and $j\geq 1$, 
then $\W$ is a Wakamatsu tilting subcategory of $\B$ provided that $l(\W)$ is a Wakamatsu tilting subcategory of $\A$.
\end{corollary}

\begin{proof}
Because $l(\W)$ is a Wakamatsu tilting subcategory of $\A$ if and only if $(\Proj(\A)$, $l(\W))$ is a Wakamatsu tilting pair in $\A$.
Since $\Proj(\A)=l(\Proj(\B))$, from Theorem \ref{thm:4.9} we know that $(\Proj(\B), \W)$ is a Wakamatsu tilting pair in $\B$.
Thus $\W$ is a Wakamatsu tilting subcategory of $\B$.
\end{proof}

We end this section with applications to right Gorenstein categories and especially Gorenstein projective modules and $\omega$-Gorenstein projective modules.

\begin{proposition} \label{prop:4.11}
Let $(\mathcal{B},\mathcal{A}, i, e, l)$ be a cleft extension and $\W$ be a self-orthogonal subcategory of $\B$.
Suppose that $F(\W)\subseteq \W$ and $\mathbb{L}_{j}F(X)=0$ for any $X\in{^\bot}\W$ and $j\geq 1$. 
If $B\in \mathbf{rG}(\W)$, then $l(B)\in \mathbf{ rG}(l(\W))$.
\end{proposition}

\begin{proof}
Since $\W$ is self-orthogonal, then $\mathbf{ rG}(\W)=\X_{\W}$ by \cite[Remark 3.1]{ZW26}.
Given $B\in \mathbf{ rG}(\W)$, it follows from Lemma \ref{lem:4.3} that $l(B)\in \X_{l(\W)}\subseteq \mathbf{ rG}(l(\W))$.
\end{proof}



\begin{corollary} \textnormal{\!\!\!(Compare \cite[Proposition 3.1]{Q25})}\label{cor:4.12}
Let $(\mathcal{B},\mathcal{A},i,e,l)$ be a cleft extension.
Suppose that $F(\Proj(\B))\subseteq \Proj(\B)$ and $\mathbb{L}_{j}F(X)=0$ for any $X\in{^\bot}\Proj(\B)$ and $j\geq 1$. 
If $B\in \GP(\B)$, then $l(B)\in \GP(\A)$.
\end{corollary}

\begin{proof}
Note that $GP(\B)=\mathbf{rG}(\Proj(\B))$.
Given $B\in \GP(\B)$, it follows from Proposition \ref{prop:4.11} that $l(B)\in \mathbf{ rG}(l(\Proj(\B)))$.
Thus $l(B)\in {^\bot}l(\Proj(\B))$, and there is an exact sequence
$$0\rightarrow l(B)\xrightarrow{l(f_{0})} l(P_{0})\xrightarrow{l(f_{1})} l(P_{1})\rightarrow\cdots $$
with each $l(P_{k})\in l(\Proj(\B))\subseteq \Proj(\A)$ and $\Img l(f_{k})\in{^\bot}l(\Proj(\B))$.
Since any projective object of $\A$ is a direct summand of $l(P)$ for some $P\in\Proj(\B)$,
hence ${^\bot}l(\Proj(\B))\subseteq {^\bot}\Proj(\A)$,
which implies that $l(B)\in \GP(\A)$.
\end{proof}

\begin{corollary} \label{coro:4.13}
Let $\omega$ be a Wakamatsu tilting left $R$-module.
Suppose that $M\otimes_{R}\omega\in \add\omega$ and $\Tor_{j}^{R}(M, X)=0$ for any $X\in{^\bot}\omega$ and $j\geq 1$. 
If $C\in \GP_{\omega}(R)$, then $l(C)\in \GP_{l(\omega)}(R\ltimes M)$.
\end{corollary}

\begin{proof}
Let $\W= \add\omega$, then $\mathbf{rG}(\add\omega)$ is exactly the subcategory $\GP_{\omega}(R)$ consisting of 
$\omega$-Gorenstein projective modules.
Given $C\in \GP_{\omega}(R)$, it follows from Proposition \ref{prop:4.11} that $l(C)\in \mathbf{rG}(l(\add\omega))$.

We claim that $l(\omega)$ is a Wakamatsu tilting left $R\ltimes M$-module.
In fact, because $\omega$ is a Wakamatsu tilting left $R$-module, 
$\add\omega$ is a Wakamatsu tilting subcategory of mod $R$.
Thus $\add l(\omega)=l(\add\omega)$ is a Wakamatsu tilting subcategory of mod $R\ltimes M$ by Corollary \ref{cor:4.6},
which yields that $l(\omega)$ is a Wakamatsu tilting left $R\ltimes M$-module, as claimed.
Therefore, $\mathbf{rG}(\add l(\omega))= \GP_{l(\omega)}(R\ltimes M)$.
\end{proof}




\section{Applications}

In this section, we will apply the foregoing results to $\theta$-extensions and tensor rings, 
which recover some known results about trivial ring extensions and triangular matrix rings.
In what follows, all rings are nonzero associative rings with identity and all modules are unitary. For a ring $R$, we write $\Mod R$ for the category of left $R$-modules. 

\begin{definition}{\rm (\cite{Ma93})}
{\rm Let $R$ be a ring, $M$ an $R$-$R$-bimodule and $\theta : M \otimes _{R} M \rightarrow M $ an
associative $R$-bimodule homomorphism. The {\it $\theta$-extension} of $R$ by $M$, denoted by
$R\ltimes _{\theta} M$, is defined to be the ring with underlying group $R \oplus M$ and multiplication
given as follows: $$(r, m)\cdot(r',m'):=(r r',r m'+mr'+\theta (m\otimes m')),$$
for any $r, r'
\in R$ and $m, m' \in M$.}
\end{definition}
Let $S:=R\ltimes _{\theta} M$.
Then a left $S$-module is identified with a pair
$(X,\alpha)$, where $X\in \Mod R$ and $\alpha \in \Hom _R(M\otimes_{R}X, X)$ such that $\alpha \circ (1_M\otimes \alpha)
=\alpha\circ(\theta \otimes 1_X)$.
Moreover, we have the following ring homomorphisms
$R\rightarrow S$ given by $r \mapsto (r, 0)$ and $S\rightarrow R$ given by $(r, m) \mapsto r$.
By \cite[Section 6.3]{KP25}, we have the following
cleft extension
of module categories
$$\xymatrix@!=8pc{ \Mod R \ar[r]|{i=_SR\otimes _R-} & \Mod S
			 \ar[r]|{e=_RS\otimes _S-} \ar@/_2pc/[l]|{q=_RR\otimes _S-} & \Mod R
			\ar@/_2pc/[l]|{l=_SS\otimes _R-} \ar@(ur,ul)_F    },  $$
where $q(X,\alpha)=\Coker \alpha$,
$i(Y)=(Y,0)$, $l(Y)=(Y\oplus M\otimes _R Y,\tiny {\left(\begin{array}{cc} 0 &0 \\ 1 & \theta \otimes 1_Y \end{array}\right)})$,
$e(X,\alpha)=X$ and $F(Y)=M\otimes _RY$.
Moreover, every cleft extension of module categories
is isomorphic to a cleft extension induced by a
$\theta$-extension, see \cite[Proposition 6.9]{KP25}.

From Theorem \ref{thm:3.5},
we have the following result, which generalizes both \cite[Theorem 4.6 and Corollary 4.7]{ZW26} and \cite[Theorem 4.3]{M22},
since a trivial ring extension $R\ltimes M$ is a $\theta$-extension with $\theta=0$.


\begin{proposition} \label{prop:5.2}
Let $M$ be an $R$-$R$-bimodule and $\C, \T$ be two subcategories of Mod $R$ 
such that $\Tor_{j}^{R}(M, T)=0$ for any $T\in\T$ and $j\geq 1$. Then the following conditions are equivalent.
\begin{enumerate}
\item $(\C, \T)$ is an $n$-tilting pair and $M\otimes_{R}\C\subseteq \C^{\bot}$, $M\otimes_{R}\T\subseteq \T^{\bot}$. 

\item $(l(\C), l(\T))$ is an $n$-tilting pair in Mod $R\ltimes _{\theta} M$.
 \end{enumerate}
\end{proposition}

We write where $M^{\otimes 0} = R$ and $M^{\otimes (k+1)} = M \otimes _R M^{\otimes k}$ for $k \geq 0$.
For an integer $m\geq 0$, recall that $M$ is said to be {\it $m$-nilpotent} if $M^{\otimes (m+1)} = 0$.

\begin{lemma} \label{lem:5.3}
Let $R\ltimes _{\theta} M$ be a $\theta$-extension.
If $M$ is nilpotent, then $l(\Proj(R))=\Proj(R\ltimes _{\theta} M)$.
\end{lemma}

\begin{proof}
Since $M$ is nilpotent, so is the functor $F=M\otimes_{R}-$. 
This implies that  the natural transformation $\eta: F^{2}\to F$ is nilpotent. 
From \cite[Lemma 2.3]{Ma93} we know that $\Ker q=0$,
and hence the assertion follows from Remark \ref{rmk:3.7}(1).
\end{proof}

\begin{corollary} \label{cor:5.4}
Let $Y$ be in $\Mod R$ such that $\Tor_{j}^{R}(M, Y)=0$ for all $j\geq 1$. If $M$ is nilpotent,
then the following are equivalent.
\begin{enumerate}
\item $Y$ is $n$-tilting and $(M\otimes_{R}Y)^{(I)}\in Y^{\bot}$ for any set $I$.

\item $l(Y)$ is an $n$-tilting $R\ltimes _{\theta} M$-module.
\end{enumerate}
\end{corollary}

\begin{proof}

The assertion is obtained from Corollary \ref{cor:3.6} and Lemma \ref{lem:5.3}.
\end{proof}

The next two results follow directly from Theorem \ref{thm:4.5} and Corollary \ref{cor:4.6}, respectively.

\begin{proposition} \label{prop:5.5}
Let $(\C, \W)$ be a Wakamatsu tilting pair in Mod $R$ and $M$ an $R$-$R$-bimodule.
If $\Tor_{j}^{R}(M, X)=0$ for any $X\in{^\bot}\W$ and $j\geq 1$ and $M\otimes_{R}\C^{\bot}\subseteq \C^{\bot}$,
$M\otimes_{R}\W\subseteq \W$,
then $(l(\C), l(\W))$ is a Wakamatsu tilting pair in Mod $R\ltimes _{\theta} M$.
\end{proposition}

\begin{corollary} \label{cor:5.6}
Let $\W$ be a Wakamatsu tilting subcategory of Mod $R$ and $M$ an $R$-$R$-bimodule.
If $l(\Proj(R))=\Proj(R\ltimes _{\theta} M)$, $M\otimes_{R}\W\subseteq \W$ and $\Tor_{j}^{R}(M, X)=0$ for any $X\in{^\bot}\W$ and $j\geq 1$,
then $l(\W)$ is a Wakamatsu tilting subcategory of Mod $R\ltimes _{\theta} M$.
\end{corollary}

From Theorem \ref{thm:4.9} and Corollary \ref{cor:4.10}, we obtain immediately the following two results, respectively.

\begin{proposition} \label{prop:5.7}
Let $\C, \W$ be subcategories of Mod $R$ and $M$ an $R$-$R$-bimodule
such that $\Tor_{j}^{R}(M, X)=0$ for any $X\in\C\cup \W$ and $j\geq 1$.
If $(l(\C), l(\W))$ is a Wakamatsu tilting pair in  Mod $R\ltimes _{\theta} M$ and $\mathbb{L}_{j}q(Y)=0$ for any $Y\in l(\C)^{\bot}\cup{^\bot}l(\W)$ and $j\geq 1$. 
Then $(\C, \W)$ is a Wakamatsu tilting pair in Mod $R$.
\end{proposition}

\begin{corollary} \label{cor:5.8}
Let $\W$ be a subcategory of Mod $R$ and $M$ an $R$-$R$-bimodule
such that $\Tor_{j}^{R}(M, X)=0$ for any $X\in\W$ and $j\geq 1$.
If $l(\Proj(R))=\Proj(R\ltimes _{\theta} M)$ and $\mathbb{L}_{j}q(Y)=0$ for any $Y\in {^\bot}l(\W)$ and $j\geq 1$,
then $\W$ is a Wakamatsu tilting subcategory of Mod $R$ provided that $l(\W)$ is a Wakamatsu tilting subcategory of Mod $R\ltimes _{\theta} M$.
\end{corollary}

\begin{remark} \label{rmk:5.9}
{\rm If $\theta=0$, from \cite[Corollary 1.6(c)]{FGR75} we know that $l(\Proj(R))=\Proj(R\ltimes M)$.
Note that for $(X,\alpha)\in$ Mod $R\ltimes M$, 
$i(R)\otimes_{R\ltimes M}(X,\alpha)\cong\Coker(\alpha)=q(X,\alpha)$ by \cite[Lemma 2.2(1)]{M25},
then $q\cong i(R)\otimes_{R\ltimes M}-$, and hence $\mathbb{L}_{j}q(Y)=\Tor_{j}^{R\ltimes M}(i(R), Y)$. 
Thus \cite[Theorem 3.3 and 3.4]{ZW26} are obtained by Corollary \ref{cor:5.6} and \ref{cor:5.8}, respectively.}
\end{remark}

Let $R$ be a ring and $N$ an $R$-$R$-bimodule.
We denote by $T_R(N) =\oplus _{k=0}^{\infty}N^{\otimes k}$ the {\it tensor ring} .
It follows from
\cite{KP25} that there is an isomorphism $T_R(N)\cong R\!\ltimes\!_{\theta}M$, 
where $M=N\oplus N^{\otimes 2}\oplus \cdots$ and $\theta$ is induced by $N^{\otimes k}\otimes_{\Gamma}N^{\otimes l}\rightarrow N^{\otimes (k+l)}$. 
Thus tensor ring is a special $\theta$-extension with $\theta \neq 0$. 
If $N$ is $m$-nilpotent, so is $M$. In this case, $l(\Proj(R))=\Proj(T_R(N))$.

Applying the above results over $\theta$-extensions to tensor rings, 
we can obtain the following results immediately.

\begin{corollary} \label{cor:5.10}
Assume that $N$ is $m$-nilpotent and $Y\in\Mod R$.
For all $1\leqslant k\leqslant m$, if $\Tor_{\geq 1}^{R}(N^{\otimes_R k}$, $Y)=0$, 
then the following are equivalent.
\begin{enumerate}
\item $Y$ is $n$-tilting and $(N^{\otimes_R k}\otimes_{R}Y)^{(I)}\in Y^{\bot}$ for any set $I$. 

\item if $l(Y)$ is an $n$-tilting $T_R(N)$-module.
\end{enumerate}
\end{corollary}

\begin{corollary} \label{cor:5.11}
Assume that $N$ is $m$-nilpotent and $\W$ is a Wakamatsu tilting subcategory of Mod $R$.
If for each $1\leqslant k\leqslant m$, $N^{\otimes_R k}\otimes_{R}\W\subseteq \W$ and $\Tor_{\geq 1}^{R}(N^{\otimes_R k}, X)=0$ for any $X\in{^\bot}\W$,
then $l(\W)$ is a Wakamatsu tilting subcategory of Mod $T_R(N)$.
\end{corollary}

\begin{corollary} \label{cor:5.12}
Assume that $N$ is $m$-nilpotent and $\W$ is a subcategory of Mod $R$. 
If $\Tor_{\geq 1}^{R}(N^{\otimes_R k}, X)=0$ for all $X\in\W$ and $1\leqslant k\leqslant m$,
and $\Tor_{\geq 1}^{T_R(N)}(i(R), Y)=0$  for all $Y\in {^\bot}l(\W)$,
then $\W$ is a Wakamatsu tilting subcategory of Mod $R$ provided that $l(\W)$ is a Wakamatsu tilting subcategory of Mod $T_R(N)$.
\end{corollary}

\begin{proof}
From \cite[Corollary 1.8]{TW25} we know that $q\cong i(R)\otimes_{T_R(N)}-$,
hence the statement follows from Corollary \ref{cor:5.8}.
\end{proof}

\section*{Acknowledgments}

This research was partially supported by the National Natural Science Foundation of China (Grant No. 12061026) and the Key Research Projects Plan of Higher Education Institutions in Henan Province (Grant No. 26B110007).


\end{document}